\documentclass[12pt]{amsart}
\usepackage[nobysame,abbrev,alphabetic]{amsrefs}
\usepackage{amssymb}
\usepackage[all]{xy}

\DeclareMathOperator{\Hom}{Hom}

\DeclareMathOperator{\Img}{Im}

\DeclareMathOperator{\Ker}{Ker}

\DeclareMathOperator{\trg}{trg}
\DeclareMathOperator{\gr}{gr}
\DeclareMathOperator{\Val}{Val}

\makeatletter
\renewcommand{\BibLabel}{%\hfill
    \Hy@raisedlink{\hyper@anchorstart{cite.\CurrentBib}\hyper@anchorend}%
    [\thebib]%
}
\makeatother

\begin{document}
\title{}
\author{}

\newtheorem{thm}{Theorem}[section]
\newtheorem{cor}[thm]{Corollary}
\newtheorem{lem}[thm]{Lemma}
\newtheorem{prop}[thm]{Proposition}
\newtheorem{defin}[thm]{Definition}
\newtheorem{exam}[thm]{Example}
\newtheorem{examples}[thm]{Examples}
\newtheorem{rem}[thm]{Remark}
\newtheorem*{thmA}{Theorem A}
\newtheorem*{thmA'}{Theorem A'}
\newtheorem*{thmB}{Theorem B}
\swapnumbers
\newtheorem{rems}[thm]{Remarks}
\newtheorem*{acknowledgment}{Acknowledgment}
\numberwithin{equation}{section}
\newcommand\Lref[1]{{Lemma~\ref{#1}}}%
\newcommand\Pref[1]{{Proposition~\ref{#1}}}%
\newcommand\Cref[1]{{Corollary~\ref{#1}}}%
\newcommand\Rref[1]{{Remark~\ref{#1}}}%
\newcommand\Tref[1]{{Theorem~\ref{#1}}}%
\newcommand\Dref[1]{{Definition~\ref{#1}}}%
\newcommand\Sref[1]{{Section~\ref{#1}}}%
\newcommand\Ssref[1]{{Subsection~\ref{#1}}}%
\newcommand{\dec}{\mathrm{dec}}
\newcommand{\dirlim}{\varinjlim}
\newcommand{\nek}{,\ldots,}
\newcommand{\proj}{{\rm proj}}
\newcommand{\inv}{^{-1}}
\newcommand{\isom}{\cong}
\newcommand{\Massey}{\mathrm{Massey}}
\newcommand{\ndiv}{\hbox{$\,\not|\,$}}
\newcommand{\pr}{\mathrm{pr}}
\newcommand{\tensor}{\otimes}
\newcommand{\U}{\mathrm{U}}
\newcommand{\ZX}{\dbZ\langle\langle X\rangle\rangle}
\newcommand{\alp}{\alpha}
\newcommand{\gam}{\gamma}
\newcommand{\Gam}{\Gamma}
\newcommand{\del}{\delta}
\newcommand{\eps}{\epsilon}
\newcommand{\lam}{\lambda}
\newcommand{\Lam}{\Lambda}
\newcommand{\sig}{\sigma}
\newcommand{\dbF}{\mathbb{F}}
\newcommand{\dbR}{\mathbb{R}}
\newcommand{\dbU}{\mathbb{U}}
\newcommand{\dbZ}{\mathbb{Z}}
\newcommand{\gra}{\mathfrak{a}}
\newcommand{\grc}{\mathfrak{c}}
\newcommand{\grd}{\mathfrak{d}}
\newcommand{\grD}{\mathfrak{D}}

\title[Filtrations of free groups]{Filtrations of free groups arising from the lower central series}

\author{Michael Chapman and Ido Efrat}
\address{Department of Mathematics\\
Ben-Gurion University of the Negev\\
P.O.\ Box 653, Be'er-Sheva 84105\\
Israel} \email{efrat@math.bgu.ac.il, michael169chapman@gmail.com}

\thanks{This work was supported by the Israel Science Foundation (grants No.\ 23/09 and 152/13).
It is partially based on the M.Sc.\ thesis of the first author done under the supervision of the second author in Ben-Gurion University of the Negev.}

\keywords{filtrations, free groups, lower central filtration, lower $p$-central filtration, $p$-Zassenhaus filtration, Magnus theory, upper-triangular unipotent representations, Massey products}

\subjclass[2010]{Primary 20F14, Secondary 20E05, 20H25, 20F40}

\begin{abstract}
We make a systematic study of filtrations of a free group $F$ defined as products of powers of the lower central series of $F$.
Under some assumptions on the exponents, we characterize these filtrations in terms of the group algebra, the Magnus algebra of non-commutative power series, and linear representations by upper-triangular unipotent matrices.
These characterizations generalize classical results of Gr\"un, Magnus, Witt, and Zassenhaus from the 1930's, as well as later results on the lower $p$-central filtration and the $p$-Zassenhaus filtrations.
We derive alternative recursive definitions of such filtrations, extending results of Lazard.
Finally, we relate these filtrations to Massey products in group cohomology.
\end{abstract}

 \dedicatory{To the memory of O.V.\ Melnikov}

\maketitle

\section{Introduction}
Given a group $G$,  we denote its \textsl{lower central filtration} by $G^{(n,0)}$, $n=1,2, \ldots$ .
Thus  $G^{(1,0)}=G$ and $G^{(n+1,0)}=[G^{(n,0)},G]$ for $n\geq1$.
Let $F=F_X$ be the free group on a finite basis $X$.
Classical results from the 1930's give the following three alternative descriptions of $F^{(n,0)}$.

\smallskip

(I)$^0$ \ \textsl{Power series description} (Magnus \cite{Magnus37}*{p.\ 111}): \quad
Let $\dbZ\langle\langle X\rangle\rangle$ be the ring of all formal power series over the set $X$ of non-commuting variables and with coefficients in $\dbZ$.
Let $\grd_\dbZ$ be the ideal in $\dbZ\langle\langle X\rangle\rangle$ generated by $X$, and define the \textsl{Magnus homomorphism} $\mu_\dbZ\colon F\to \dbZ\langle\langle X\rangle\rangle^\times$ by $\mu_\dbZ(x)=1+x$ for $x\in X$.
Then $F^{(n,0)}=\mu_\dbZ\inv(1+\grd_\dbZ^n)$.

\smallskip

(II)$^0$ \ \textsl{Group algebra description} (\cite{Magnus37} and Witt \cite{Witt37}): \quad
 $F^{(n,0)}=F\cap(1+\grc_\dbZ^n)$, where  $\grc_\dbZ$ is the augmentation ideal in the group algebra $\dbZ[F]$.

\smallskip

(III)$^0$ \ \textsl{Description by unipotent matrices} (Gr\"un \cite{Grun36}; see also Magnus \cite{Magnus35}): \quad
$F^{(n,0)}$ is the intersection of all kernels of homomorphisms $F\to \dbU_n(\dbZ)$, where
$\dbU_n(\dbZ)$ is the group of all upper-triangular unipotent $n\times n$ matrices over $\dbZ$
(this is somewhat implicit in Gr\"un's paper, and we refer to R\"ohl \cite{Rohl85} for a more modern exposition, where this is shown for \textsl{lower}-triangular matrices).

For a detailed history of these results see \cite{ChandlerMagnus82}.
These descriptions of $F^{(n,0)}$ have natural analogs in the context of free profinite groups, where one considers free profinite groups, complete group rings, and continuous homomorphisms.

\medskip

Next, for a prime number $p$, one defines the \textsl{lower $p$-central filtration} $G^{(n,p)}$, $n=1,2\nek$ of a group $G$ by
\[
G^{(n,p)}=\prod_{i=1}^n(G^{(i,0)})^{p^{n-i}}.
\]
Alternatively, it has the following inductive definition  \cite{NeukirchSchmidtWingberg}*{Prop.\ 3.8.6}:
\[
G^{(1,p)}=G, \qquad G^{(n+1,p)}=(G^{(n,p)})^p[G^{(n,p)},G] \hbox{ \ for } n\geq1.
\]
For a free group $F=F_X$ this filtration has the following analogs of (I)$^0$--(III)$^0$:

\smallskip

(I)$^p$ \ As shown by Skopin \cite{Skopin50} (for $p\neq2$) and Koch \cite{Koch60},
\[
F^{(n,p)}=\mu_\dbZ\inv(1+(p\dbZ\langle\langle X\rangle\rangle+\grd_\dbZ)^n).
\]

(II)$^p$ \
By a result of Koch \cite{Koch02}*{Th.\ 7.14} (who however works in a pro-$p$ context),
\[
F^{(n,p)}=F\cap(1+(p\dbZ[F]+\grc_\dbZ)^n).
\]

(III)$^p$ \ $F^{(n,p)}$ is the intersection of all kernels of homomorphisms $F\to\dbU_{d+1}(\dbZ/p^{n-d}\dbZ)$, where $1\leq d\leq n-1$ (Min\'a\v c and T\^an \cite{MinacTan15}*{\S2});
Alternatively, $F^{(n,p)}$ is the intersection of all kernels of homomorphisms $F\to G(n,p)$, where $G(n,p)$ is the group of all upper-triangular unipotent $n\times n$-matrices $(c_{ij})$ over $\dbZ/p^n\dbZ$ such that $c_{ij}\in p^{j-i}\dbZ/p^n\dbZ$ for every $i\leq j$ \cite{Efrat14b}.

\medskip

A third well-studied filtration of a group $G$ of a similar nature is the \textsl{$p$-Zassenhaus filtration} $G_{(n,p)}$, $n=1,2\nek$ where $p$ is a prime number (see \cite{Zassenhaus39}).
Here
\[
G_{(n,p)}=\prod_{i=1}^n(G^{(i,0)})^{p^{\lceil\log_p(n/i)\rceil}}=\prod_{ip^j\geq n}(G^{(i,0)})^{p^j}.
\]
By a result of Lazard, it has the following alternative inductive definition (see \cite{Lazard65}*{p.\ 209, (3.14.5)},  \cite{DixonDuSautoyMannSegal99}*{Th.\ 11.2}):
\[
G_{(1,p)}=G, \qquad G_{(n,p)}=(G_{(\lceil n/p\rceil,p}))^p\prod_{i+j=n}[G_{(i,p)},G_{(j,p)}] \hbox{ \  for } n\geq2.
\]
For this filtration one has the following analogs of (I)$^0$--(III)$^0$:

\smallskip

(I)$_p$ \
In the pro-$p$ case $F_{(n,p)}=\mu_{\dbF_p}\inv(1+\grd_{\dbF_p}^n)$, where  $\mu_{\dbF_p}\colon F\to\dbF_p\langle\langle X\rangle\rangle^\times$ and $\grd_{\dbF_p}$ are defined similarly to (I)$^0$
(see Morishita \cite{Morishita12}*{\S8.3}), Vogel \cite{Vogel05}*{Th.\ 2.19(ii)}, or \cite{Efrat14a}*{Prop.\ 6.2}).

\smallskip

(II)$_p$ \
One has  $G_{(n,p)}=G\cap (1+\grc_{\dbF_p}^n)$, where $\grc_{\dbF_p}$ denotes the augmentation ideal in the group $\dbF_p$-algebra $\dbF_p[G]$ of $G$.
This was proved by Jennings and Brauer \cite{Jennings41}*{Th.\ 5.5} for a finite $p$-group $G$,
and was extended to arbitrary groups by Lazard \cite{Lazard54}*{Cor.\ 6.10} (see also Quillen \cite{Quillen68}*{Th.\ 2.4}).

\smallskip

(III)$_p$ \ $F_{(n,p)}$ is the intersection of all kernels of homomorphisms $F\to \dbU_n(\dbF_p)$, in both the discrete and profinite contexts (see \cite{Efrat14a}*{Th.\ A}, \cite{Efrat14b}*{Example 6.4}, and \cite{MinacTan15}*{Th.\ 2.6(a)}).

\medskip

In this paper we investigate systematically general filtrations of free groups, obtained as products of powers of the lower central series.
We develop a general framework for studying such filtrations, and prove structure theorems which contain most of
the above-mentioned facts as special cases.

More specifically, given a map $e\colon \{(n,i)\in\dbZ^2\ |\ 1\leq i\leq n\}\to\dbZ_{\geq0}$ we consider the subgroups $\prod_{i=1}^n(F^{(i,0)})^{e(n,i)}$ of $F$.
We also consider the ideal $\grd_\dbZ^{(e,n)}=\sum_{i=1}^ne(n,i)\grd_\dbZ^i$ of $\dbZ\langle\langle X\rangle\rangle$, and the ideal $\grc_\dbZ^{(e,n)}=\sum_{i=1}^ne(n,i)\grc_\dbZ^i$ of $\dbZ[F]$.
We first prove in Theorem \ref{first main theorem} that, under certain assumptions on $e$,
\[
\prod_{i=1}^n(F^{(i,0)})^{e(n,i)}=\mu_\dbZ\inv(1+\grd_\dbZ^{(e,n)})=F\cap(1+\grc_\dbZ^{(e,n)}).
\]
Special choices of $e$ then give (I)$^0$, (I)$^p$, (II)$^0$, (II)$^p$, and new variants of (I)$_p$, (II)$_p$.
Next we show in Theorem \ref{intersection thm} that (again, under certain assumptions on $e$), the group $\mu_\dbZ\inv(1+\grd_\dbZ^{(e,n)})$ is the intersection of all kernels of homomorphisms $\varphi\colon F\to \dbU_{d+1}(\dbZ/e(n,d)\dbZ)$, with $1\leq d\leq n-1$.
This generalizes  (III)$^0$, (III)$^p$, (III)$_p$.

Finally, in \S\ref{section on inductive definitions}--\S\ref{section on Zassenhaus filtration} we study inductive definitions of the above subgroups $\prod_{i=1}^n(G^{(i,0)})^{e(n,i)}$, as we had for the motivating examples $G^{(n,0)}$, $G^{(n,p)}$, $G_{(n,p)}$.
After developing some general machinery in \S\ref{section on inductive definitions}, we generalize in \S\ref{section on A filtration} the inductive definition of the lower central series  $G^{(n,0)}$, and lower $p$-central series $G^{(n,p)}$, to the following situation: for a sequence $A=(a_i)_{i=1}^\infty$ of non-negative integers we define the \textsl{$A$-filtration} $G^{(n,A)}$, $n=1,2\nek$ of $G$ by $G^{(1,A)}=G$ and $G^{(n+1,A)}=(G^{(n,A)})^{a_n}[G^{(n,A)},G]$ for $n\geq1$.
We then prove in Theorem \ref{closed formula for A filtration} that
\[
G^{(n,A)}=\prod_{i=1}^n(G^{(i,0)})^{e(n,i)},
 \]
where $e(n,i)$ is the gcd of all products of $n-i$ elements from $a_1\nek a_{n-1}$ (and $e(n,n)=1$).
Note that taking $a_i=0$ or $a_i=p$ for every $i$, recovers the inductive definitions of $G^{(n,0)},G^{(n,p)}$, respectively.

In \S\ref{section on Zassenhaus filtration} we use the general machinery of \S\ref{section on inductive definitions} to extend Lazard's description of the $p$-Zassenhaus filtration, by showing that for a $p$-power $q$ one has
\[
G_{(n,q)}=\prod_{i=1}^n(G^{(i,0)})^{q^{\lceil\log_p(n/i)\rceil}},
 \]
 where the \textsl{$q$-Zassenhaus filtration} $G_{(n,q)}$ is defined inductively by $G^{(1,q)}=1$ and $G_{(n,q)}=(G_{(\lceil n/p\rceil,q)})^q\prod_{i+j=n}[G_{(i,q)},G_{(j,q)}]$ (Theorem \ref{q Zassenhaus as product}).

\medskip

Finally, in \S\ref{section on Massey products} we investigate these filtrations from a cohomological viewpoint, and relate them to Massey products in $H^2$.
As a new contribution to the classical theory of Magnus, Witt and Zassenhaus, we apply our techniques to compute the subgroup of $H^2(F/F^{(n,0)},\dbZ)$ generated by the $n$-fold Massey products corresponding to words in the alphabet $X$.
Specifically, we show that it is a free $\dbZ$-module whose rank is the necklace function of $n$ and $|X|$; see Corollary \ref{rank of H2} for details.

\medskip

While throughout the paper we work in the context of discrete groups, many of our main results have quite straightforward profinite  analogs.
The study of such filtrations in the profinite context has become increasingly interesting in recent years in infinite Galois theory -- see for instance  \cite{BogomolovTschinkel12}, \cite{CheboluEfratMinac12}, \cite{EfratMinac11a}, \cite{EfratMinac11b}, \cite{MinacTan15}, \cite{Morishita04},  \cite{Topaz12}, \cite{Vogel05}.

\medskip

We thank Ilya Tyomkin for valuable remarks which shortened some of our original arguments.
We also thank the anonymous referee for his very helpful comments, which we have used to improve the presentation at several points.

\section{Graded Lie algebras}
\label{section on garded Lie algebras}
We fix a set $X$.
For a commutative unital ring $R$ let $R\langle X\rangle$ be the free associative $R$-algebra on $X$, i.e., the algebra of non-commuting polynomials in elements of  $X$ and with coefficients in $R$.
It is graded by total degree, and we denote by $R\langle X\rangle^{(n)}$ its homogenous component of degree $n$.
Thus $R\langle X\rangle\isom\bigoplus_{n=0}^\infty R\langle X\rangle^{(n)}$ as additive groups.
We also consider the ring $R\langle\langle X\rangle\rangle=\prod_{n=0}^\infty R\langle X\rangle^{(n)}$.
When $X$ is finite, $R\langle\langle X\rangle\rangle$ is the ring of formal power series over the set $X$ of non-commuting variables and coefficients in $R$.
Given $\alp\in R\langle X\rangle$ or $\alp\in R\langle\langle X\rangle\rangle$, we denote the homogenous component of $\alp$ of degree $n$ by $\alp^{(n)}$.

Let $\grd_R=\langle X\rangle$ be the ideal  in $R\langle\langle X\rangle\rangle$ generated by $X$.
We observe:

\begin{lem}
\label{inverse in Magnus ring}
If $\alp\in\grd_R$, then $\sum_{k=0}^\infty\alp^k$ is a well-defined element of $1+\grd_R$, and is the unique two-sided inverse of $1-\alp$.
Consequently, if $\gra\subseteq\grd_R$ is an ideal of $R\langle\langle X\rangle\rangle$, then $1+\gra$ is a subgroup of $R\langle\langle X\rangle\rangle^\times$.
\end{lem}

As with any associative algebra, we consider $R\langle X\rangle$ and $R\langle\langle X\rangle\rangle$ also as Lie $R$-algebras with the Lie brackets $[\alp,\beta]=\alp\beta-\beta\alp$.
This makes  $R\langle X\rangle$ a graded Lie $R$-algebra.

Let $L_{X,R}$ be the free Lie $R$-algebra on $X$.
It has a natural grading \cite{SerreLie}*{p.\ 19}, and we denote the $n$-th homogenous component of $L_{X,R}$ by $L_{X,R}^{(n)}$.
The universal property of $L_{X,R}$ gives rise to a natural graded Lie algebra homomorphism $\phi_R\colon L_{X,R}\to R\langle X\rangle$ which is the identity on $X$.
The algebra $R\langle X\rangle$ may be identified with the universal algebra of the Lie algebra $L_{X,R}$ via this map, and $L_{X,R}$ is a free $R$-module \cite{SerreLie}*{Part I, Ch.\ IV, Th.\ 4.2(1)(3)}.

Now let $F=F_X$ be the free group on the basis $X$, and $F^{(n,0)}$, $n=1,2\nek$ its lower central filtration of $F$ (see the Introduction).
The quotients $F^{(n,0)}/F^{(n+1,0)}$ are commutative.
Therefore we have a graded $\dbZ$-module $\gr(F)=\bigoplus_{n=1}^\infty F^{(n,0)}/F^{(n+1,0)}$.
Moreover, the commutator map induces on $\gr(F)$ a structure of a graded Lie $\dbZ$-algebra.
The identity map on $X$ induces a graded Lie algebra homomorphism $L_{X,\dbZ}\to\gr(F)$.
It is surjective in degree $1$, and by induction, in all degrees
(in fact, it is an isomorphism \cite{SerreLie}*{Part I, Ch.\ IV, Th.\ 6.1}, but we will not need this fact).

We write $R[F]$ for the group $R$-algebra of $F=F_X$.
Let $\grc_R$ be its augmentation ideal, i.e., the ideal generated by all elements of the form $g-1$ with $g\in F$.
We identify $F$ as a subset of $R[F]$.

\begin{lem}
\label{subgroup}
Let $\grc$ be an ideal in $R[F]$.
Then $F\cap(1+\grc)$ is a subgroup of $F$.
\end{lem}
\begin{proof}
The identity   $\alp\inv=1-\alp\inv(\alp-1)$ shows that $F\cap(1+\grc)$ is closed under inversion.
The rest is immediate.
\end{proof}

For every $x\in X$ the element $1+x$ of $R\langle\langle X\rangle\rangle$ is invertible, by Lemma \ref{inverse in Magnus ring}.
Let
\[
\mu_R\colon F\to1+\grd_R\leq R\langle\langle X\rangle\rangle^\times, \quad x\mapsto 1+x \hbox{ for } x\in X,
\]
be the {\it Magnus homomorphism}.
It extends canonically to an $R$-algebra homomorphism $\mu_R\colon R[F]\to R\langle\langle X\rangle\rangle$.
We note that
\begin{equation}
\label{c0 and d0}
\mu_R(\grc_R)\subseteq\grd_R.
\end{equation}
There are inclusions
\begin{equation}
\label{inclusions}
F^{(n,0)}\subseteq F\cap(1+\grc_R^n)\subseteq \mu_R\inv(1+\grd_R^n).
\end{equation}
For $R=\dbZ$, the classical results of Magnus and Witt (\cite{Magnus35}, \cite{Magnus37}, \cite{Witt37}) show that these are equalities:
\begin{equation}
\label{MagnusWitt}
F^{(n,0)}=F\cap(1+\grc_\dbZ^n)=\mu_\dbZ\inv(1+\grd_\dbZ^n).
\end{equation}

For every $n\geq 1$ we have a group homomorphism $\mu_R^{(n)}\colon F^{(n,0)}/F^{(n+1,0)}\to R\langle\langle X\rangle\rangle^{(n)}=R\langle X\rangle^{(n)}$ given by $\mu_R^{(n)}(gF^{(n+1,0)})=\mu_R(g)^{(n)}$ for $g\in F^{(n,0)}$.
For $g\in F^{(n,0)}$ and $h\in F^{(m,0)}$ we have as in  \cite{SerreLie}*{p.\ 25, ($*$)}, $[g,h]\in F^{(n+m,0)}$ and
\[
\mu_R([g,h])^{(n+m)}=\mu_R(g)^{(n)}\mu_R(h)^{(m)}-\mu_R(h)^{(m)}\mu_R(g)^{(h)}.
\]
Therefore the group homomorphisms $\mu_R^{(n)}$ combine to a graded Lie $R$-algebra homomorphism
\[
\gr\mu_R=\bigoplus_{n=1}^\infty\mu_R^{(n)}\colon\gr(F)\to R\langle X\rangle.
\]
The triangle of graded Lie $\dbZ$-algebra homomorphisms
\begin{equation}
\label{triangle}
\xymatrix{
L_{X,\dbZ}\ar[r]^{\phi_\dbZ}\ar@{->>}[d]_{}  &  \dbZ\langle X\rangle \\
\gr(F) \ar[ru]_{\gr\mu_\dbZ}
}
\end{equation}
commutes on every $x\in X$, and hence on all of $L_{X,\dbZ}$.

Let  $\Delta_R\colon R\langle X\rangle\to R\langle X\rangle\tensor_RR\langle X\rangle$ be the $R$-algebra homomorphism induced by the diagonal embedding $\lam\mapsto(\lam,\lam)$ for $\lam\in L_{X,R}$ (compare \cite{SerreLie}*{Part I, Ch.\ III, Def.\ 5.1}).
When $R$ is  torsion-free as a $\dbZ$-module,  \cite{SerreLie}*{Part I, Ch.\ III, Th.\ 5.4} gives an exact sequence of $R$-modules
\begin{equation}
\label{exact seq}
L_{X,R}\xrightarrow{\phi_R} R\langle X\rangle\to R\langle X\rangle\tensor_R R\langle X\rangle,
\end{equation}
where the right homomorphism is given by  $\lam\mapsto  \Delta_R(\lam)-\lam\tensor1+1\tensor\lam$.

\begin{lem}
\label{torsion-free}
For $R=\dbZ$, the cokernel of $\gr\mu_\dbZ\colon \gr(F)\to\dbZ\langle X\rangle$ is torsion-free as a $\dbZ$-module.
\end{lem}
\begin{proof}
By (\ref{triangle}), this cokernel coincides with the cokernel $\dbZ\langle X\rangle/\phi_\dbZ(L_{X,\dbZ})$ of $\phi_\dbZ$.
By (\ref{exact seq}), the latter cokernel is a submodule of the torsion-free $\dbZ$-module  $\dbZ\langle X\rangle\tensor_\dbZ \dbZ\langle X\rangle$.
Hence it is also torsion-free.
\end{proof}

\begin{cor}
\label{from g to h}
Suppose that either $R=\dbZ$ or $R$ is a field.
Then the cokernel of $\mu_R^{(n)}\colon F^{(n,0)}/F^{(n+1,0)}\to R\langle X\rangle^{(n)}$ is a torsion-free $R$-module.
\end{cor}
\begin{proof}
The case $R=\dbZ$ follows from Lemma \ref{torsion-free}.
The case where $R$ is a field follows from the general structure theory of $R$-linear spaces.
\end{proof}

\section{Multiplicatively descending maps}

\begin{defin}
\label{def: multiplicatively descending map}
\rm
We say that a map
\[
e\colon \{(n,i)\in\dbZ^2\ |\ 1\leq i\leq n\}\to\dbZ_{\geq0}
\]
is \textsl{multiplicatively descending} if it satisfies the following conditions:
\begin{enumerate}
\item[(i)]
$e(n,n)=1$ for every $n$;
\item[(ii)]
$e(n,i)\in e(n,i+1)\dbZ$ for every $1\leq i<n$.
\end{enumerate}
\end{defin}

\begin{exam}
\label{trivial example}
\rm
The \textsl{trivial} multiplicatively descending map is defined by  $e(n,i)=0$ for $1\leq i<n$, and $e(n,n)=1$.
\end{exam}

\begin{exam}
\label{e for A}
\rm
Let $(a_i)_{i=1}^\infty$ be a sequence of non-negative integers.
We define a multiplicatively descending map by setting for $1\leq i<n$,
\[
e(n,i)=\gcd\Bigl\{\prod_{j\in J}a_j\ \Bigm|\ J\subseteq\{1,2\nek n-1\},\ \ |J|=n-i\Bigr\},
\]
and setting $e(n,n)=1$.
\end{exam}

\begin{exam}
\label{lower q-central series}
\rm
In the previous example take the constant sequence $(a)_{i=1}^\infty$, where $a$ is a non-negative integer.
We obtain that $e(n,i)=a^{n-i}$  is a multiplicatively descending map.
For $a=0$ (and $e(n,n)=1$) this recovers the trivial multiplicatively descending map.
\end{exam}

\begin{exam}
\label{log pni}
\rm
Let $t$ be a positive integer and $p$ a prime number.
For integers $1\leq i\leq n$, let $j(n,i)=\left\lceil{\log_p\left(\frac ni\right)}\right\rceil$.
Then $e(n,i)=p^{tj(n,i)}$ is a multiplicatively descending map.
Indeed, (i) is immediate, and condition (ii) follows from $j(n,i)\geq j(n,i+1)$.
\end{exam}

\begin{defin}
\label{binomial}
\rm
We call a multiplicatively descending map $e$ \textsl{binomial} if  for every positive integers $n,i,l$ such that $l\leq e(n,i)$ and $il\leq n$ one has $\binom{e(n,i)}l\in e(n,il)\dbZ$.
\end{defin}

For example, the trivial multiplicatively descending map is clearly binomial.
A useful way to verify this property in more involved situations is given by the next lemma.
For a prime number $p$ let $v_p$ be the $p$-adic valuation on $\dbZ$.

\begin{lem}
\label{condition  (iii)}
Let $e$ be a multiplicatively descending map satisfying the following condition:
\begin{enumerate}
\item[(iii)]
For every positive integers $n,i,r$ and a prime number $p$ such that $ip^r\leq n$, if $v_p(e(n,i))\geq r$, then $v_p(e(n,i))-r\geq v_p(e(n,ip^r))$.
\end{enumerate}
Then $e$ is binomial.
\end{lem}
\begin{proof}
Let $n,i,l$ be as in Definition \ref{binomial}.
It suffices to show that for every prime number $p$,
\begin{equation}
\label{p-adic valuations}
v_p\Biggl(\binom{e(n,i)}l\Biggr)\geq v_p(e(n,il)).
\end{equation}
To this end let $r=v_p(l)$ and $s=v_p(e(n,i))$.

When $s<r$ we have $ip^s<ip^r\leq il\leq n$, so by (iii) (with $r$ replaced by $s$), $v_p(e(n,ip^s))=0$.
As $ip^s< il$, (ii) implies that  $v_p(e(n,il))=0$, and (\ref{p-adic valuations}) holds.

Next suppose that $s\geq r$.
For $a\geq b\geq1$ one has $\binom ab=\frac ab\binom{a-1}{b-1}$,
whence $v_p(\binom ab)\geq v_p(a)-v_p(b)$.
As $ip^r\leq il\leq n$, conditions (iii), and (ii), respectively, therefore give
\[
v_p\Biggl(\binom{e(n,i)}l\Biggr)\geq s-r\geq v_p(e(n,ip^r))\geq v_p(e(n,il))
\]
in this case as well.
\end{proof}

\begin{exam}
\label{e for A binom}
\rm
The map of Example \ref{e for A} is binomial.
Indeed, take a sequence $(a_i)_{i=1}^n$ of non-negative integers, a prime number $p$, and $1\leq i\leq n-1$.
We first show that if $v_p(e(n,i))\geq1$, then $v_p(e(n,i+1))<v_p(e(n,i))$.
Indeed, there exists a subset $J$ of $\{1,2\nek n-1\}$ of size $n-i$ such that $v_p(e(n,i))=v_p(\prod_{j\in J}a_j)$.
Further, there exists $j_0\in J$ such that $v_p(a_{j_0})\geq1$.
Then
\[
v_p(e(n,i))>v_p\bigl(\prod_{j\in J\setminus\{j_0\}}a_j\bigr) \geq v_p(e(n,i+1)).
\]
Iterating this, we obtain that if $v_p(e(n,i))\geq r\geq0$ and $i+r\leq n$, then $v_p(e(n,i))-r\geq v_p(e(n,i+r))$.
By (ii), this implies that $ v_p(e(n,i))-r\geq v_p(e(n,ip^r))$ whenever $ip^r\leq n$.
This proves (iii).
We now apply Lemma \ref{condition (iii)}.
\end{exam}

\begin{exam}
\label{log pni binom}
\rm
The map of Example \ref{log pni} is binomial.
Indeed, consider a prime number $p$ and integers $t,n,i,r\geq1$.
Suppose that $v_p(e(n,i))\geq r$.
Therefore
\[
v_p(e(n,i))-r=t\Bigl\lceil\log_p\frac ni\Bigr\rceil-r\geq t\Bigl\lceil\log_p\frac n{ip^r}\Bigr\rceil=v_p(e(n,ip^r)).
\]
Thus condition (iii) holds, and we use again Lemma \ref{condition  (iii)}.
\end{exam}

\section{Powers of the lower central series}
\label{section on powers of lcs}
As before, let $X$ be a fixed set, let $R$ be a unital commutative ring, and let $F=F_X$ be the free group on basis $X$.
For a multiplicatively descending map $e$, we consider the ideals
\[
\grc_R^{(e,n)}=\sum_{i=1}^ne(n,i)\grc_R^i, \qquad \grd_R^{(e,n)}=\sum_{i=1}^ne(n,i)\grd_R^i
\]
of $R[F]$, $R\langle\langle X\rangle\rangle$, respectively.
We record the following alternative description of $\grd_R^{(e,n)}$:

\begin{lem}
\label{description of grd(e,n)}
One has $\grd_R^{(e,n)}=\bigcap_{d=1}^n(e(n,d)\grd_R+\grd_R^{d+1})$.
\end{lem}
\begin{proof}
We denote the ideal $\bigcap_{d=1}^n(e(n,d)\grd_R+\grd_R^{d+1})$ by $J$.

Let $1\leq i,d\leq n$.
If $1\leq i\leq d$, then $e(n,i)\in e(n,d)\dbZ$, so $e(n,i)\grd_R^i\subseteq e(n,d)\grd_R$.
If $d<i\leq n$, then $e(n,i)\grd_R^i\subseteq \grd_R^{d+1}$.
Thus in both cases, $e(n,i)\grd_R^i\subseteq e(n,d)\grd_R+\grd_R^{d+1}$.
Since $i,d$ were arbitrary,  $\grd_R^{(e,n)}\subseteq J$.

Conversely, we show by induction on $1\leq m\leq n$ that
\begin{equation}
\label{inclusion of J}
J\subseteq \sum_{i=1}^{m-1}e(n,i)\grd_R^i+\grd_R^m.
\end{equation}
For $m=1$ this is trivial.
Assume that the induction hypothesis holds for $1\leq m\leq n-1$.
We note that
\[
\grd_R^m\cap J\subseteq \grd_R^m\cap(e(n,m)\grd_R+\grd_R^{m+1})=e(n,m)\grd_R^m+\grd_R^{m+1}.
\]
As $\sum_{i=1}^{m-1}e(n,i)\grd_R^i\subseteq\grd_R^{(e,n)}\subseteq J$, (\ref{inclusion of J}) implies that
\[
\begin{split}
J&\subseteq \sum_{i=1}^{m-1}e(n,i)\grd_R^i+(\grd_R^m\cap J)
\subseteq  \sum_{i=1}^{m-1}e(n,i)\grd_R^i+e(n,m)\grd_R^m+\grd_R^{m+1}\\
&=\sum_{i=1}^me(n,i)\grd_R^i+\grd_R^{m+1},
\end{split}
\]
completing the induction.

For $m=n$ we obtain  $J\subseteq\sum_{i=1}^ne(n,i)\grd_R^i=\grd_R^{(e,n)}$, whence $\grd_R^{(e,n)}=J$.
\end{proof}

\begin{cor}
\label{cor to th 4.2}
One has $\grd_\dbZ^{(e,n)}\cap\grd_\dbZ^m\subseteq e(n,m)\grd_\dbZ^m+\grd_\dbZ^{m+1}$.
\end{cor}
\begin{proof}
By Lemma \ref{description of grd(e,n)},
\[
\begin{split}
\grd_\dbZ^{(e,n)}\cap\grd_\dbZ^m
&\subseteq (e(n,m)\grd_\dbZ+\grd_\dbZ^{m+1})\cap\grd_\dbZ^m\\
&=(e(n,m)\grd_\dbZ\cap\grd_\dbZ^m)+\grd^{m+1}=e(n,m)\grd_\dbZ^m+\grd_\dbZ^{m+1}.
\end{split}
\]
\end{proof}

\begin{thm}
\label{first main theorem}
Let $e$ be a binomial multiplicatively descending map.
Then
\[
\prod_{i=1}^n(F^{(i,0)})^{e(n,i)}\subseteq F\cap(1+\grc_R^{(e,n)})\subseteq\mu_R\inv(1+\grd_R^{(e,n)}).
\]
Moreover, when $R=\dbZ$, these are equalities.
\end{thm}
\begin{proof}
The right inclusion follows from (\ref{c0 and d0}).

For the left inclusion, let $1\leq i\leq n$ and let $g\in F^{(i,0)}$.
By (\ref{inclusions}), $g=1+\alpha$ with $\alpha\in\grc_R^i$.
We have
\[
g^{e(n,i)}=1+\sum_{l=1}^{e(n,i)}{{e(n,i)}\choose l}\alpha^l.
\]
Let  $1\leq l\leq e(n,i)$.
Since $e$ is binomial, if $il\leq n$, then
\[
{{e(n,i)}\choose l}\alpha^l\in e(n,il)\grc_R^{il}\subseteq\grc_R^{(e,n)}.
\]
If $il\geq n$, then
\[
{{e(n,i)}\choose l}\alpha^l\in\grc_R^{il}\subseteq\grc_R^n=e(n,n)\grc_R^n\subseteq\grc_R^{(e,n)}
\]
in this case as well.
Therefore $g^{e(n,i)}\in1+\grc_R^{(e,n)}$.
It remains to recall that $F\cap(1+\grc_R^{(e,n)})$ is a subgroup of $F$ (Lemma \ref{subgroup}).

Finally, we prove that when $R=\dbZ$, the set
\[
\Gam=\mu_\dbZ\inv(1+\grd_\dbZ^{(e,n)})\setminus\prod_{i=1}^n(F^{(i,0)})^{e(n,i)}
\]
is empty.
Indeed, assume that $\Gam\neq\emptyset$.
Since $e(n,n)=1$ we have $\Gam\cap F^{(n,0)}=\emptyset$.
Therefore for every $g\in \Gam$ there is an integer $m(g)$ such that $1\leq m(g)<n$ and $g\in F^{(m(g),0)}\setminus F^{(m(g)+1,0)}$.
We may choose $g\in\Gam$ with $m:=m(g)$ maximal.
The definition of $\Gamma$ implies that $\mu_\dbZ(g)\in1+\grd_\dbZ^{(e,n)}$, and $g\in F^{(m,0)}$ implies that $\mu_\dbZ(g)\in1+\grd_\dbZ^m$, by (\ref{MagnusWitt}).
It therefore follows from Corollary \ref{cor to th 4.2} that $\mu_\dbZ(g)^{(m)}=e(n,m)\lambda\neq0$ for some $\lambda\in \dbZ\langle X\rangle^{(m)}$.
By Corollary \ref{from g to h}, the cokernel of $\mu_\dbZ^{(m)}$ is a torsion-free $\dbZ$-module.
Hence there exists $h\in F^{(m,0)}$ with $\mu_\dbZ(h)^{(m)}=\lambda$.
We obtain that  $gh^{-e(n,m)}\in F^{(m,0)}$ and $\mu_\dbZ^{(m)}(gh^{-e(n,m)})=0$, so by (\ref{MagnusWitt}), $gh^{-e(n,m)}\in F^{(m+1,0)}$.

On the other hand, by the first part of the theorem, $h^{-e(n,m)}\in\mu_\dbZ\inv(1+\grd_\dbZ^{(e,n)})$.
Since $1+\grd_\dbZ^{(e,n)}$ is multiplicatively closed (Lemma \ref{inverse in Magnus ring}), $gh^{-e(n,m)}\in\mu_\dbZ\inv(1+\grd_\dbZ^{(e,n)})$.
Further, $h^{e(n,m)}\in\prod_{i=1}^n(F^{(i,0)})^{e(n,i)}$, so we have $gh^{-e(n,m)}\not\in\prod_{i=1}^n(F^{(i,0)})^{e(n,i)}$.
Therefore $gh^{-e(n,m)}\in\Gam$, contrary to the maximality of $m$.
\end{proof}

\begin{exam}
\label{examples for first theorem - trivial e}
\rm
Let $R=\dbZ$ and let $e(n,i)$ be the trivial multiplicatively descending map (Example \ref{trivial example}).
Then Theorem \ref{first main theorem} contains the classical results (\ref{inclusions}),  (\ref{MagnusWitt}) of Magnus and Witt as special cases (note however that (\ref{MagnusWitt}) was used in the proof of Theorem \ref{first main theorem}).
\end{exam}

\begin{exam}
\label{examples for first theorem - A filtration}
\rm
Let $R=\dbZ$, let $a$ be a non-negative integer, let $e(n,i)=a^{n-i}$ as in Example \ref{lower q-central series}, and recall that it is binomial (Example \ref{e for A binom}).
We obtain from Theorem \ref{first main theorem} that
\[
\prod_{i=1}^n(F^{(i,0)})^{a^{n-i}}
=F\cap(1+\sum_{i=1}^na^{n-i}\grc_\dbZ^i)
=\mu_\dbZ\inv(1+\sum_{i=1}^na^{n-i}\grd_\dbZ^i).
\]
The sequence $\prod_{i=1}^n(F^{(i,0)})^{a^{n-i}}$, $n=1,2\nek$ is a generalization of the descending $p$-lower sequence of $F$ (when one takes $a=p$).
In this sense, Theorem  \ref{first main theorem} for this choice of $e$ generalizes the results of Skopin and Koch (\cite{Skopin50}, \cite{Koch60}, \cite{Koch02}*{Th.\ 7.14}) discussed in the Introduction  (the latter result is however in a pro-$p$ context).
\end{exam}

\begin{exam}
\label{examples for first theorem - Zassenhaus}
\rm
Let $R=\dbZ$, and take a prime number $p$ and an integer $t\geq1$.
Consider the multiplicatively descending map $e$ of Example \ref{log pni}, and recall that it is binomial (Example \ref{log pni binom}).
We obtain from Theorem \ref{first main theorem} that
\[
\prod_{i=1}^n(F^{(i,0)})^{p^{t\lceil\log_p\frac ni\rceil}}
=F\cap(1+\sum_{i=1}^np^{t\lceil\log_p\frac ni\rceil}\grc_\dbZ^i)
=\mu_\dbZ\inv(1+\sum_{i=1}^np^{t\lceil\log_p\frac ni\rceil}\grd_\dbZ^i).
\]
As we will later show in \S\ref{section on Zassenhaus filtration}, for $t=1$ the product $\prod_{i=1}^n(F^{(i,0)})^{p^{\lceil\log_p\frac ni\rceil}}$ is the $n$th term $F_{(n,p)}$ of the $p$-Zassenhaus filtration of $F$.
\end{exam}

\begin{rem}
\rm
The proof of Theorem \ref{first main theorem} does not use condition (ii) of Definition  \ref{def: multiplicatively descending map}.
However, given a map $e\colon\{(n,i)\ |\ 1\leq i\leq n\}\to\dbZ_{\ge0}$ satisfying only $e(n,n)=1$, we may define a map $e'$ by $e'(n,i)=\gcd_{1\leq j\leq i}e(n,j)$.
Then $e'$ is a multiplicatively descending map.
Clearly, $e(n,i)\in e'(n,i)\dbZ$ for every $i\leq n$.
Furthermore, $e'(n,i)$ is a linear combination of $e(n,1)\nek e(n,i)$ with integral coefficients.
Therefore
\[
\prod_{i=1}^n(F^{(i,0)})^{e(n,i)}=\prod_{i=1}^n(F^{(i,0)})^{e'(n,i)}, \quad \grc_R^{(e,n)}=\grc_R^{(e',n)}, \quad \grd_R^{(e,n)}=\grd_R^{(e',n)}.
\]
Hence we lose nothing by assuming (ii) here.
\end{rem}

\section{Intersections of kernels}
\label{section on intersections of kernels}
Let $R$ be again a unital commutative ring and $n\geq2$ an integer.
Let $\dbU_n(R)$ be the group of $n\times n$ upper-triangular unipotent matrices over $R$.
For an integer $t\geq0$ let $T_{n,t}(R)$ be the set of $n\times n$ matrices $(a_{ij})$ over $R$ with $a_{ij}=0$ for every $1\leq i,j\leq n$ such that $j-i<t$.
Thus $T_{n,t}(R)=\{0\}$ for $n\leq t$.
We view $T_{n,0}(R)$ as an $R$-algebra, filtered by the powers of the ideal $T_{n,1}(R)$.
Note that $T_{n,t}(R)T_{n,t'}(R)\subseteq T_{n,t+t'}(R)$, and in particular  $T_{n,1}(R)^n=\{0\}$.
We write $I_n$ for the identity matrix of order $n\times n$.

As before, let $F=F_X$ be the free group on a set $X$ of generators.
Let $\varphi\colon F\to\dbU_n(R)$ be a group homomorphism.
We filter $R\langle\langle X\rangle\rangle$ by the powers of $\grd_R$.
Its universal property then gives rise to a unique  $R$-algebra homomorphism $\hat\varphi\colon R\langle\langle X\rangle\rangle\rightarrow T_{n,0}(R)$ which is compatible with the filtrations, and such that $\hat\varphi(x)=\varphi(x)-I_n\in T_{n,1}(R)$ for $x\in X$.
Thus  $\varphi(\grd_R)\subseteq T_{n,1}(R)$.

For $1\leq i<j\leq n$ let $\varphi_{ij}\colon F\to R$ be the composition of $\varphi$ with the projection on the $(i,j)$-entry of $\dbU_n(R)$.

Also let $X^*$ be the free monoid on $X$, i.e., the set of finite words in the alphabet $X$.
We denote the length of a word $w\in X^*$ by $|w|$.
Given $g\in F$ we write
\[
\mu_R(g)=\sum_{w\in X^*}\mu_{R,w}(g)w,
\]
where $\mu_{R,w}$ is the coefficient of the Magnus homomorphism $\mu_R\colon F\to R\langle\langle X\rangle\rangle^\times$ at $w$ (see \S\ref{section on garded Lie algebras}).
The fact that $\mu_R$ is a group homomorphism implies, as in \cite{Efrat14a}*{Lemma 7.5}, the following lemma:

\begin{lem}
\label{intlem}
Let $w=(x_1\cdots x_{n-1})\in X^*$ be a word of length $n-1$.
The map $\varphi_{R,w}\colon F\rightarrow\dbU_n(R)$, defined by $(\varphi_{R,w})_{ij}=\mu_{R,(x_i\cdots x_{j-1})}$ for $1\leq i<j\leq n$, is a group homomorphism.
\end{lem}

\begin{prop}
\label{inclusion of intersections}
For every positive integer $n$ we have
\[
\mu_R\inv(1+\grd_R^n)=
\bigcap\Bigl\{\Ker(\varphi) \Bigm| \varphi\in\Hom(F,\dbU_n(R))\Bigr\}
=\bigcap_{{w\in X^*}\atop{|w|=n-1}}\Ker(\varphi_{R,w}).
\]
\end{prop}
\begin{proof}
Take $\varphi\in\Hom(F,\dbU_n(R))$ and let $\hat\varphi\colon R\langle\langle X\rangle\rangle\to T_{n,0}(R)$ be an $R$-algebra homomorphism as above.
Then $\hat\varphi(\grd_R^n)\subseteq T_{n,1}(R)^n=\{0\}$.
Since $\varphi=\hat\varphi\circ\mu_R$ on $F$, this implies that $\mu_R\inv(1+\grd_R^n)\subseteq\Ker(\varphi)$.

The middle intersection is trivially contained in the right intersection.

Finally, take $g$ in the right intersection.
Every word $u\in X^*$ of length $k$, with $1\leq k\leq n-1$, can be extended to a word $w\in X^*$ of length $n-1$ with prefix $u$.
Then $\mu_{R,u}(g)=(\varphi_{R,w}(g))_{1,k+1}=0$.
We conclude that $g\in \mu_R\inv(1+\grd_R^n)$.
\end{proof}

We can now add to Theorem \ref{first main theorem} the following characterization of $\mu_R\inv(1+\grd_R^{(e,n)})$ in terms of linear representations by upper-triangular unipotent matrices.

\begin{thm}
\label{intersection thm}
Let $e$ be multiplicatively descending map.
Then
\[
\begin{split}
\mu_R\inv(1+\grd_R^{(e,n)})&=\bigcap_{d=1}^{n-1}\bigcap\Bigl\{\Ker(\varphi)\ \Bigm|\ \varphi\in\Hom(F,\dbU_{d+1}(R/e(n,d)R)\Bigr\} \\
&=\bigcap_{d=1}^{n-1}\bigcap_{{w\in X^*}\atop{|w|=d}}\Ker(\varphi_{R/e(n,d)R,w}).
\end{split}
\]
\end{thm}
\begin{proof}
By Lemma \ref{description of grd(e,n)},
$\grd_R^{(e,n)}=\bigcap_{d=1}^n(e(n,d)\grd_R+\grd_R^{d+1})$.
Therefore
\[
\mu_R\inv(1+\grd_R^{(e,n)})=\bigcap_{d=1}^n\mu_{R/e(n,d)R}\inv(1+\grd_{R/e(n,d)R}^{d+1}).
\]
The Theorem now follows from Proposition \ref{inclusion of intersections}.
\end{proof}

\begin{exam}
\rm
Let $e(n,i)$ be the trivial multiplicatively descending map (Example \ref{trivial example}).
Then, by Proposition \ref{inclusion of intersections},
\begin{equation}
\label{yyy}
\mu_R\inv(1+\grd_R^{(e,n)})=\mu_R\inv(1+\grd_R^n)=\bigcap\Bigl\{\Ker(\varphi)\ |\ \varphi\in\Hom(F,\dbU_n(R))\Bigr\}.
\end{equation}

In particular, when $R=\dbZ$, the group $\mu_\dbZ\inv(1+\grd_\dbZ^{(e,n)})$ is the $n$-th term $F^{(n,0)}$ of the lower central series of $F$ (see (\ref{MagnusWitt})).
Then (\ref{yyy}) is essentially due to Gr\"un \cite{Grun36} (see also \cite{Rohl85}; note that Gr\"un works with \textsl{lower-}triangular unipotent matrices).
\end{exam}

\begin{exam}
\rm
Let $a$ be a non-negative integer and consider the multiplicatively descending map $e(n,i)=a^{n-i}$ (Example \ref{lower q-central series}).
We obtain that
\[
\mu_\dbZ\inv(1+\grd_\dbZ^{(e,n)})=\bigcap_{d=1}^{n-1}\bigcap\Bigl\{\Ker(\varphi)\ \Bigm|\ \varphi\in\Hom(F,\dbU_{d+1}(\dbZ/a^{n-d}\dbZ)\Bigr\}.
\]
When $a=p$ is a prime number, $\mu_\dbZ\inv(1+\grd_\dbZ^{(e,n)})$ is the $n$-th term $F^{(n,p)}$ in the lower $p$-central filtration of $F$, by the results of Skopin and Koch (see Example \ref{examples for first theorem - A filtration}).
Thus, in this special case, the first equality in Theorem \ref{intersection thm} recovers the characterization of $F^{(n,p)}$ due to Min\'a\v c and T\^an \cite{MinacTan15}*{Th.\ 2.7(c)} mentioned in (III)$^p$ of the Introduction
(an equivalent characterization was independently proved in \cite{Efrat14b}).
\end{exam}

\begin{exam}
\rm
Let $R=\dbZ$, let $t$ be a positive integer, and let $p$ be a prime number.
Consider the multiplicatively descending map $e(n,i)=p^{tj(n,i)}$ of Example \ref{log pni}, where $j(n,i)=\lceil\log_p\bigl(\frac ni\bigr)\rceil$.
Then $\mu_\dbZ\inv(1+\grd_\dbZ^{(e,n)})$ is the $n$th term $F_{(n,p)}$ in the $p$-Zassenhaus filtration (see  Example \ref{examples for first theorem - Zassenhaus}).
We obtain from Theorem \ref{intersection thm} that
\[
F_{(n,p)}=\bigcap_{d=1}^{n-1}\Bigl\{\Ker(\varphi)\ |\ \varphi\in\Hom(F,\dbU_{d+1}(\dbZ/p^{tj(n,d)}\dbZ))\Bigr\}.
\]
\end{exam}

\begin{exam}
\rm
Let $R=\dbF_p$ for a prime number $p$ and let $e(n,i)$ be the trivial multiplicatively descending map.
Then $\mu_{\dbF_p}\inv(1+\grd_{\dbF_p}^{(e,n)})=\mu_{\dbF_p}\inv(1+\grd_{\dbF_p}^n)$ is again the $n$-th term $F_{(n,p)}$ of the $p$-Zassenhaus filtration of $F$ (see (I)$_p$ in the Introduction).
Hence Proposition \ref{inclusion of intersections} gives
\begin{equation}
\label{Zassenhaus as kernel intersection}
F_{(n,p)}=\bigcap\{\Ker(\varphi)\ |\ \varphi\in\Hom(F,\dbU_n(\dbF_p))\}.
\end{equation}
A profinite version of (\ref{Zassenhaus as kernel intersection})  was proved in \cite{Efrat14a}*{Th.\ A} as a special case of a more general result on Massey products in profinite cohomology.
More direct alternative proofs were later given in \cite{Efrat14b}*{Ex.\ 6.4} (also in the discrete setting) and \cite{MinacTan15}*{Th.\ 2.7}.
\end{exam}

We conclude this section with some terminology and facts that will be needed in \S\ref{section on Massey products}.
Let $Z_n(R)$ be the subgroup of $\dbU_n(R)$ consisting of all unipotent matrices which are zero except for the main diagonal and (possibly) entry $(1,n)$.
It lies in the center of $\dbU_n(R)$, so we may define
\[
\bar\dbU_n(R)=\dbU_n(R)/Z_n(R).
\]
We may view the elements of  $\bar\dbU_n(R)$ as upper-triangular unipotent $n\times n$-matrices without the $(1,n)$-entry.

Given a word $w=(x_1\cdots x_n)$, let $\bar\varphi_w\colon F\to\bar \dbU_{n+1}(R)$ be the homomorphism induced by $\varphi_w$.
The following proposition complements Proposition \ref{inclusion of intersections}.

\begin{prop}
\label{kernel intersection to bar U}
For $n\geq1$ we have $\mu_R\inv(1+\grd_R^n)=\bigcap_{{w\in X^*}\atop{|w|=n}}\Ker(\bar\varphi_w)$.
\end{prop}
\begin{proof}
For $x_1,x_2\nek x_{n-1},x_n\in X$ one has inclusions
\[
\Ker(\varphi_{(x_1\cdots x_{n-1})})\supseteq \Ker(\bar\varphi_{(x_1\cdots x_n)})=\Ker(\varphi_{(x_1\cdots x_{n-1})})\cap \Ker(\varphi_{(x_2\cdots x_n)}).
\]
It follows that
\[
\bigcap_{{w\in X^*}\atop{|w|=n-1}}\Ker(\varphi_w)=\bigcap_{{w\in X^*}\atop{|w|=n}}\Ker(\bar\varphi_w).
\]
Now apply Proposition \ref{inclusion of intersections}.
\end{proof}

It follows from Proposition \ref{kernel intersection to bar U} that, for every $w\in X^*$ of length $n\geq1$,
the restriction of $\varphi_w$ to $\mu_R\inv(1+\grd_R^n)$ is into $Z_{n+1}(R)$.
We therefore obtain:

\begin{cor}
\label{restriction of mu hom}
The map $\mu_R\inv(1+\grd_R^n)\to R$,  $g\mapsto \varphi_w(g)_{1,n+1}$, is a group homomorphism.
\end{cor}

\section{Inductive definitions of filtrations}
\label{section on inductive definitions}
We now turn to study inductive definitions of the general filtrations discussed in \S\ref{section on powers of lcs},
similarly to what we had for the lower $p$-central series and $p$-Zassenhaus filtrations.
First we develop in the current section a general machinery for such inductive constructions.
In \S\ref{section on A filtration} and \S\ref{section on Zassenhaus filtration} we apply it to obtain the above known examples, and extend them to new cases.

Suppose that $T\subseteq\dbZ_{\ge1}^2$, and let  $f\colon \dbZ_{\geq2}\to\dbZ_{\geq0}$, $g\colon \dbZ_{\geq2}\to\dbZ_{\geq1}$ be maps such that $g(n)<n$ for every $n\geq2$.
For a group $G$ we define inductively subgroups $G_{(n)}$, $n=1,2\nek$ by
\begin{equation}
\label{recursive definition}
G_{(1)}=G, \quad
G_{(n)}=G_{(g(n))}^{f(n)}\prod_{(s,t)\in T,\ s+t=n}[G_{(s)},G_{(t)}].
\end{equation}
The subgroups $G_{(n)}$ are characteristic, whence normal in $G$.
Note that the sequence $G_{(n)}$, $n=1,2\nek$ need not be decreasing.

Let $e(n,i)$ be a multiplicatively descending map, and assume that:
\begin{enumerate}
\item[(1)]
For every $(s,t)\in T$,  $1\leq i\leq s$ and $1\leq j\leq t$ one has
$e(s,i)e(t,j)\in e(s+t,i+j)\dbZ$.
\item[(2)]
For every $n\geq2$ and $1\leq j_1\nek j_l\leq g(n)$ such that $1\leq l\leq f(n)$ and $j_1+\cdots+j_l\leq n$ one has
\[
\binom{f(n)}l e(g(n),j_1)\cdots e(g(n),j_l)\in e(n,j_1+\cdots+j_l)\dbZ.
\]
\end{enumerate}

\begin{rem}
\label{meaning of (1), (2)}
\rm
Condition (1) implies that  $\grd_\dbZ^{(e,s)}\grd_\dbZ^{(e,t)},\grd_\dbZ^{(e,t)}\grd_\dbZ^{(e,s)}\subseteq\grd_\dbZ^{(e,s+t)}$ for $(s,t)\in T$.
Condition (2) implies that $\binom{f(n)}l\alp^l\in \grd_\dbZ^{(e,n)}$ for every $n\geq2$, $1\leq l\leq f(n)$, and $\alp\in \grd_\dbZ^{(e,g(n))}$.
\end{rem}

\begin{thm}
\label{rec sequence inclusion}
Let $F=F_X$ be a free group.
Then $F_{(n)}\subseteq\mu_\dbZ\inv(1+\grd_\dbZ^{(e,n)})$ for every $n$.
\end{thm}
\begin{proof}
We argue by induction on $n$.
For $n=1$ we have $\grd_\dbZ^{(e,1)}=\grd_\dbZ$, so $F_{(1)}=F=\mu_\dbZ\inv(1+\grd_\dbZ^{(e,1)})$.

Suppose that $n>1$.
For every $\alp\in \grd_\dbZ^{(e,g(n))}=\sum_{j=1}^{g(n)}e(g(n),j)\grd_\dbZ^j$ and $1\leq l\leq f(n)$, assumption (2) implies that $\binom{f(n)}l\alp^l\in\grd_\dbZ^{(e,n)}$ (see Remark \ref{meaning of (1), (2)}).
Therefore $(1+\alp)^{f(n)}=\sum_{l=0}^{f(n)}\binom{f(n)}l\alp^l\in 1+\grd_\dbZ^{(e,n)}$.
By the induction hypothesis, this shows that $F_{(g(n))}^{f(n)}\subseteq\mu_\dbZ\inv(1+\grd_\dbZ^{(e,n)})$.

Next  take $(s,t)\in T$ with $s+t=n$ and consider $\alp\in\grd^{(e,s)}$ and $\beta\in \grd^{(e,t)}$.
By (1),  $\grd_\dbZ^{(e,s)}\grd_\dbZ^{(e,t)},\grd_\dbZ^{(e,t)}\grd_\dbZ^{(e,s)}\subseteq\grd_\dbZ^{(e,n)}$ (see Remark \ref{meaning of (1), (2)}).
As $1+\alp,1+\beta\in\ZX^\times$ (Lemma \ref{inverse in Magnus ring}), we obtain that
\[
\begin{split}
[1+\alp,1+\beta]&=1+(1+\alp)\inv(1+\beta)\inv\bigl((1+\alp)(1+\beta)-(1+\beta)(1+\alp)\bigr)\\
&=1+(1+\alp)\inv(1+\beta)\inv(\alp\beta-\beta\alp)\subseteq 1+\grd_\dbZ^{(e,n)}.
\end{split}
\]
From this and the induction hypothesis we conclude that $[F_{(s)},F_{(t)}]\subseteq \mu\inv(1+\grd_\dbZ^{(e,n)})$.
Consequently, $F_{(n)}\subseteq \mu\inv(1+\grd_\dbZ^{(e,n)})$, completing the induction.
\end{proof}

\section{The $A$-filtration}
\label{section on A filtration}
Let $A=(a_i)_{i=1}^\infty$ be a sequence of non-negative integers.
Let $e$ be the multiplicatively descending map of Example \ref{e for A}.
Thus for $1\leq i<n$ we set
\[
e(n,i)=\gcd\Bigl\{\prod_{j\in J}a_j\ \Bigm|\ J\subseteq\{1,2,\nek n-1\},\ \ |J|=n-i\Bigr\},
\]
and by convention, $e(n,n)=1$.
In this section we give an inductive definition for the subgroups $\prod_{i=1}^n(G^{(i,0)})^{e(n,i)}$ of a group $G$.
Namely, we show that they coincide with the \textit{$A$-filtration}  $G^{(n,A)}$, $n=1,2\nek$ of $G$  (Theorem \ref{closed formula for A filtration}).
Recall that we defined $G^{(1,A)}=G$, and
\[
G^{(n,A)}=(G^{(n-1,A)})^{a_{n-1}}[G^{(n-1,A)},G],
\]
for every $n>1$.
For this we use the general framework of  \S\ref{section on inductive definitions} with \
\[
f(n)=a_{n-1},  \quad g(n)=n-1, \quad T=\dbZ_{\ge1}\times\{1\}.
\]

\begin{lem}
\label{(1) (2) for A-filtration}
Conditions (1) and (2) of \S\ref{section on inductive definitions} hold in this setup.
\end{lem}
\begin{proof}
For condition (1) observe that $e(s,i)e(1,1)=e(s,i)\in e(s+1,i+1)\dbZ$ for every $1\leq i\leq s$.

For condition (2), take $1\leq j_1\nek j_l\leq n-1$ with $1\leq l\leq a_{n-1}$ and $j_1+\cdots+j_l\leq n$.
We need to show that
\[
\binom{a_{n-1}}le(n-1,j_1)\cdots e(n-1,j_l)\in e(n,j_1+\cdots+j_l)\dbZ.
\]
When $l=1$ we have in fact $e(n-1,j_1)\in e(n,j_1)\dbZ$.
When $2\leq l\leq a_{n-1}$ we have $j_1+1\leq j_1+\cdots+j_l$, so
\[
e(n-1,j_1)\cdots e(n-1,j_l)\in e(n-1,j_1)\dbZ\subseteq e(n,j_1+1)\dbZ\subseteq e(n,j_1+\cdots+j_l)\dbZ.
\qedhere
\]
\end{proof}

Next we recall some basic facts about commutators of subgroups.

\begin{lem}
\label{comm}
Let $H$ be a normal subgroup of $G$ and let $a,b$ be non-negative integers.
Then:
\begin{enumerate}
\item[(a)]
$[H^a,G]\equiv [H,G]^a\pmod{[[H,G],G]}$.
\item[(b)]
$(H^a[H,G])^b[H^a[H,G],G]=(H^b[H,G])^a[H^b[H,G],G]$.
\end{enumerate}
\end{lem}
\begin{proof}
(a) \quad
For $h\in H$ and $g\in G$ we have
\[
[h^a,g]=h^{-a}(g\inv hg)^a=h^{-a}(h[h,g])^a\equiv  [h,g]^a\pmod{[[H,G],G]}.
\]
(b) \quad
Since both subgroups contain the normal subgroup $[[H,G],G]$, we prove the equality modulo $[[H,G],G]$.
By (a),
\[
(H^a[H,G])^b\equiv H^{ab}[H,G]^b\equiv H^{ab}[H^b,G] \pmod{[[H,G],G]}.
\]
Also, for $h\in H$ and $k\in [H,G]$ we have
\[
[h^ak,g]=k\inv[h^a,g]k[k,g]\in k\inv[H^a,G]k[[H,G],G]=[H^a,G][[H,G],G],
\]
since $[H^a,G]$ is normal in $G$.
Hence
\[
[H^a[H,G],G]\equiv[H^a,G]\pmod{[[H,G],G]}.
\]
Therefore the left hand side of (b) is  $H^{ab}[H^b,G][H^a,G][[H,G],G]$.
By symmetry, this is also  the right hand side of (b).
\end{proof}

\begin{lem}
\label{ABseq}
Let  $\sig\in S_{n-1}$ be a permutation, and $A_\sig$ the sequence obtained from $A$ by applying $\sig$ to the first $n-1$ elements.
Then $G^{(n,A)}=G^{(n,A_\sig)}$.
\end{lem}
\begin{proof}
Every permutation is a composition of transpositions, so we can assume that $\sig$ transposes $k-1$ and $k$, where $2\leq k<n$.
We apply Lemma \ref{comm}(b) with $H=G^{(k-1,A)}=G^{(k-1,A_\sig)}$,
$a=a_{k-1}$,  and $b=a_k$,  to obtain that
\[
\begin{split}
G^{(k+1,A)}&=(G^{(k,A)})^b[G^{(k,A)},G]=(H^a[H,G])^b[H^a[H,G],G] \\
&=(H^b[H,G])^a[H^b[H,G],G]=(G^{(k,A_\sig)})^a[G^{(k,A_\sig)},G]=G^{(k+1,A_\sig)}.
\end{split}
\]
Hence also $G^{(n,A)}=G^{(n,A_\sig)}$.
\end{proof}

We now obtain the main result of this section:

\begin{thm}
\label{closed formula for A filtration}
Let $A=(a_i)_{i=1}^\infty$ be a sequence of non-negative integers and $G$ a group.
Let $e(n,i)$ be as in Example \ref{e for A}.
Then for every $n\geq1$ we have
\[
G^{(n,A)}=\prod_{i=1}^n(G^{(i,0)})^{e(n,i)}.
\]
\end{thm}
\begin{proof}
Let $T$ be a subset of $\{1,2\nek n-1\}$ of size $n-i$.
There exists $\sigma\in S_{n-1}$ such that $T=\{\sigma(i),\cdots,\sigma(n-1)\}$.
Let $A_\sigma$ be again the sequence obtained from $A$ by applying $\sig$ to the first $n-1$ elements.
By the definition of the filtrations, for every $1\leq i\leq n$ we have $G^{(i,0)}\subseteq G^{(i,A_\sigma)}$.
Therefore $(G^{(i,0)})^{\prod_{j\in J}a_j}\subseteq G^{(n,A_\sigma)}$, and by Lemma \ref{ABseq}, $G^{(n,A_\sigma)}=G^{(n,A)}$.
Further, $e(n,i)$ is a linear combination of the products $\prod_{j\in T}a_j$ with integral coefficients.
We conclude that $(G^{(i,0)})^{e(n,i)}\subseteq G^{(n,A)}$, whence  $\prod_{i=1}^n(G^{(i,0)})^{e(n,i)}\subseteq G^{(n,A)}$.

For the opposite inclusion let $F=F_X$ be a free group on $X$.
By Lemma \ref{(1) (2) for A-filtration} and  Theorem \ref{rec sequence inclusion},
$F^{(n,A)}\subseteq\mu_\dbZ\inv(1+\grd_\dbZ^{(e,n)})$.
Since $e$ is binomial (Example \ref{e for A binom}), Theorem \ref{first main theorem} shows that $\mu_\dbZ\inv(1+\grd_\dbZ^{(e,n)})=\prod_{i=1}^n(F^{(i,0)})^{e(n,i)}$, so $F^{(n,A)}\subseteq\prod_{i=1}^n(F^{(i,0)})^{e(n,i)}$.
Since every group is an epimorphic image of a free group, this completes the proof.
\end{proof}

In particular, for a non-negative integer $a$, we define the \textsl{$a$-lower central series} $G^{(n,a)}$, $n=1,2\nek$ of $G$ by
\[
G^{(1,a)}=G, \quad G^{(n,a)}=(G^{(n-1,a)})^a[G^{(n-1,a)},G]
\]
for $n\geq2$.
Note that when $a=0$ (resp., $a=p$ is prime) this is just the lower central (resp., $p$-central) series, and therefore there is no conflict of notation.
We obtain:

\begin{cor}
For every integer $n\geq1$ we have
$G^{(n,a)}=\prod_{i=1}^n(G^{(i,0)})^{a^{n-i}}$.
\end{cor}

When $a=q$ is a prime power, this is contained in \cite{NeukirchSchmidtWingberg}*{Prop.\ 3.8.6}.

\section{The $q$-Zassenhaus filtration}
\label{section on Zassenhaus filtration}
In this section, let $p$ be a prime number, $t$  a positive integer, and $q=p^t$ a $p$-power.
Let $e(n,i)$ be the multiplicatively descending map of Example \ref{log pni}.
Thus for integers $1\leq i\leq n$ we set $j(n,i)=\left\lceil{\log_p\left(\frac ni\right)}\right\rceil$ and $e(n,i)=q^{j(n,i)}$.
We give an inductive construction of the subgroups $\prod_{i=1}^n(G^{(i,0)})^{e(n,i)}$ of a group $G$ in this case.
More specifically, we show that they coincide with the \textsl{$q$-Zassenhaus filtration} $G_{(n,q)}$, $n=1,2\nek$ defined inductively by
\[
G_{(1,q)}=G, \quad G_{(n,q)}=G_{(\lceil n/p\rceil,q)}^q\prod_{s+t=n}[G_{(s,q)},G_{(t,q)}].
\]

\begin{lem}
\label{Zassenhaus is decreasing}
$G_{(n-1,q)}\geq G_{(n,q)}$ for every $n\geq2$.
\end{lem}
\begin{proof}
We may assume inductively that $G=G_{(1,q)}\geq\cdots\geq G_{(n-1,q)}$.
We show that all the factors in the definition of $G_{(n,q)}$ are contained in $G_{(n-1,q)}$.

We first note that $\lceil (n-1)/p\rceil\leq\lceil n/p\rceil\leq n-1$.
By the induction hypothesis, $G_{(\lceil (n-1)/p\rceil,q)}\geq G_{(\lceil n/p\rceil,q)}$, whence
$G_{(n-1,q)}\geq G_{(\lceil (n-1)/p\rceil,q)}^q\geq G_{(\lceil n/p\rceil,q)}^q$.

Next, consider $s,t\geq1$ with $s+t=n$.
If, say $s\geq2$, then
\[
G_{(n-1,q)}\geq[G_{(s-1,q)},G_{(t,q)}]\geq[G_{(s,q)},G_{(t,q)}],
\]
and similarly when $t\geq2$.
Finally, when $s=t=1$ and $n=2$ the latter inclusion is trivial.
\end{proof}

We consider the general framework of \S\ref{section on inductive definitions} with
\[
f(n)=q,  \quad g(n)=\lceil n/p\rceil,  \quad T=\dbZ_{\geq1}^2.
\]

\begin{lem}
\label{1, 2 for q-Zassenhaus}
Conditions (1) and (2) of \S\ref{section on inductive definitions} hold in this setup.
\end{lem}
\begin{proof}
For (1), take integers $1\leq i\leq s$ and $1\leq j\leq t$.
As $st/ij\geq (s+t)/(i+j)$ we have
\[
\Bigl\lceil\log_p\frac si\Bigr\rceil+\Bigl\lceil\log_p\frac tj\Bigr\rceil
\geq \Bigl\lceil\log_p\frac {st}{ij}\Bigr\rceil
\geq \Bigl\lceil\log_p\frac {s+t}{i+j}\Bigr\rceil,
\]
whence $e(s,i)e(t,j)\in e(s+t,i+j)\dbZ$.

For (2), let $1\leq j_1\nek j_l\leq g(n)$ where $1\leq l\leq q$ and $j_1+\cdots+j_l\leq n$.

When $p\leq l$ we have $lg(n)\geq n$, so (1) implies that
\[
e(g(n),j_1)\cdots e(g(n),j_l)\in e(lg(n),j_1+\cdots+j_l)\dbZ\subseteq e(n,j_1+\cdots+j_l)\dbZ.
\]

When $1\leq l<p$ the equality $q\binom{q-1}{l-1}=l\binom ql$ shows that $q$ divides $q\choose l$.
Further,  $n/(j_1+\cdots+j_l)\leq n/j_1\leq pg(n)/j_1$, so
\[
\Bigl\lceil\log_p\frac n{j_1+\cdots+j_l}\Bigr\rceil
\leq 1+\Bigl\lceil\log_p\frac {g(n)}{j_1}\Bigr\rceil
\leq1+\Bigl\lceil\log_p\frac {g(n)}{j_1}\Bigr\rceil+\cdots+\Bigl\lceil\log_p\frac {g(n)}{j_l}\Bigr\rceil.
\]
Therefore $qe(g(n),j_1)\cdots e(g(n),j_l)\in e(n,j_1+\cdots+j_l)\dbZ$, and we get (2) in this case as well.
\end{proof}

We now prove the main result of this section.
It extends Lazard's result (\cite{Lazard65}*{p.\ 209, (3.14.5)},  \cite{DixonDuSautoyMannSegal99}*{Th.\ 11.2}) mentioned in the Introduction, which is the case $t=1$, $q=p$ of our theorem.
Note that our method of proof is different from Lazard's.

\begin{thm}
\label{q Zassenhaus as product}
Let $G$ be a group.
Then $G_{(n,q)}=\prod_{i=1}^n(G^{(i,0)})^{e(n,i)}$.
\end{thm}
\begin{proof}
For every $1\leq i\leq n$ we have by definition $G^{(i,0)}\subseteq G_{(i,q)}$ and $(G_{(i,q)})^q\subseteq G_{(ip,q)}$.
For $j=j(n,i)$ we have $ip^j\geq n$.
Using these observations and Lemma \ref{Zassenhaus is decreasing} we obtain
\[
(G^{(i,0)})^{q^j}\subseteq(G_{(i,q)})^{q^j}\subseteq G_{(ip^j,q)}\subseteq  G_{(n,q)}.
\]
This shows that $\prod_{i=1}^n(G^{(i,0)})^{e(n,i)}\subseteq G_{(n,q)}$.

For the opposite inclusion we may assume that $G=F_X$ is a free group.
By Lemma \ref{1, 2 for q-Zassenhaus} and Theorem \ref{rec sequence inclusion}, $(F_X)_{(n,q)}\subseteq\mu_\dbZ\inv(1+\grd_\dbZ^{(e,n)})$,
and by Example \ref{examples for first theorem - Zassenhaus}, the latter group is $\prod_{i=1}^n(F^{(i,0)})^{e(n,i)}$.
\end{proof}

\section{Massey products}
\label{section on Massey products}
We now use the machinery of \S\ref{section on intersections of kernels} to investigate the cohomology of the quotients $F/\mu_R\inv(1+\grd_R^n)$ for a free group $F$.
We show that the quotients of consecutive subgroups $\mu_R\inv(1+\grd_R^n)$ are dual to a certain subgroup of $H^2(F/\mu_R\inv(1+\grd_R^n),R)$ generated by $n$-fold Massey products (Theorem \ref{duality}).
The discussion is partly inspired by \cite{Efrat14a}.

Throughout this section we fix a commutative unital ring $R$.
Let $G$ be a group acting from the left on $R$, considered as an abelian group.
We write $C^i(G,R)$ for the group of (inhomogeneous) $i$-cochains on $G$ with values in $R$ and $\partial\colon C^i(G,R)\to C^{i+1}(G,R)$ for the differential.
Let $H^i(G,R)$ be the $i$th cohomology group.
We refer, e.g., to \cite{Weibel94}*{Ch.\ 6} for the basic notions in group cohomology.

Let $F=F_X$ be as before the free group on basis $X$.
We let it act trivially on $R$.
One has $H^2(F,R)=0$ \cite{Weibel94}*{Cor.\ 6.2.7}.
Let $N$ be a normal subgroup of $F$ contained in $[F,F]$.
Then every homomorphism $F\to R$ factors via $F/N$, i.e., the inflation map $H^1(F/N,R)\to H^1(F,R)$ is surjective.
The 5-term sequence in group cohomology \cite{Weibel94}*{6.8.3} therefore shows that the transgression map
$\trg\colon H^1(N,R)^F\to H^2(F/N,R)$ (which is the $d_2^{0,1}$-differential in the Hochschild--Serre spectral sequence) is an isomorphism.
We may therefore define a bilinear map
\begin{equation}
\label{cohomological perfect pairing}
(\cdot,\cdot)'\colon \ N\times H^2(F/N,R)\to R, \quad (g,\alp)'=(\trg\inv(\alp))(g).
\end{equation}
It has a trivial right kernel (see \cite{EfratMinac11b}*{(3.3)}).

Next we recall  a few notions and facts about Massey products in group cohomology; see e.g.  \cite{Kraines66}, \cite{Dwyer75}, \cite{Fenn83} for more details.
Let $n\geq2$.
An array
\[
M=\{c_{ij}\ |\ 1\leq i<j\leq n+1,\ (i,j)\neq(1,n+1)\}
\]
in $C^1(G,R)$ is called a \textsl{defining system} if $\partial c_{ij}=-\sum_{k=i+1}^{j-1}c_{ik}\cup c_{kj}$ for every $i,j$ as above.
In particular, $\partial c_{i,i+1}=0$ for every $1\leq i\leq n$.
One can show that $\sum_{k=2}^nc_{1k}\cup c_{k,n+1}$ is a $2$-cocycle (see \cite{Fenn83}*{p.\ 233}, \cite{Kraines 66}*{p.\ 432}; note that various sources have different sign conventions).
Its cohomology class $\Val(M)$ in $H^2(G,M)$ is the \textsl{value} of the defining system $M$.
Given $\chi_1\nek\chi_n\in H^1(G,M)$, the \textsl{$n$-fold Massey product}  $\langle\chi_1\nek\chi_n\rangle$ is the subset of $H^2(G,M)$ consisting of all values of defining systems $M$ as above such that $c_{i,i+1}$ is a representative of $\chi_i$, $i=1,2\nek n$.
%In this case one also writes $\langle\chi_1\nek\chi_n\rangle_M$ for the value $\Val(M)$ of $M$.
For example, when $n=2$ the $2$-fold Massey product $\langle\chi_1,\chi_2\rangle$ consists only of the cup product $\chi_1\cup\chi_2$.

By  Dwyer \cite{Dwyer75}*{Th.\ 2.4}, when $G$ acts trivially on $R$ there is a bijective correspondence between defining systems $M=(c_{ij})$ and group homomorphisms $\bar\varphi\colon G\to\bar\dbU_{n+1}(R)$ (see \S\ref{section on intersections of kernels}), given by $\bar\varphi(g)_{ij}=(-1)^{j-i}c_{ij}(g)$ for $1\leq i<j\leq n+1$ with $(i,j)\neq(1,n+1)$ (the other entries being obvious).

For the rest of this section we abbreviate
\[
F_{(n,R)}=\mu_R\inv(1+\grd_R^n).
\]
Given a word $w=(x_1\cdots x_n)\in X^*$ of length $n$, let $\varphi_w=\varphi_{R,w}\colon F\to\dbU_{n+1}(R)$ and $\bar\varphi_w\colon F\to\bar \dbU_{n+1}(R)$ be as in \S\ref{section on intersections of kernels}.
By Proposition \ref{kernel intersection to bar U}, $\bar\varphi_w$ factors via a homomorphism $\bar\varphi'_w\colon F/F_{(n,R)}\to \bar\dbU_{n+1}(R)$.
It gives rise as above to a defining system $M=(\bar c_{ij})$ in $C^1(F/F_{(n,R)},R)$.
We set $\psi_w=\Val(M)$.
Thus $\psi_w\in \langle[\bar c_{12}]\nek[\bar c_{n,n+1}]\rangle\subseteq H^2(F/F_{(n,R)},R)$, where $[\bar c_{ij}]$ denote the cohomology class of $c_{ij}$ in $H^1(F/F_{(n,R)},R)$.
For $1\leq i<j\leq n+1$ let  $c_{ij}\in C^1(F,R)$ be the projection of $\varphi_w$ on the $(i,j)$-entry.

\begin{lem}
\label{psi as transgression}
$\psi_w=\trg[c_{1,n+1}|_{F_{(n,R)}}]$.
\end{lem}
\begin{proof}
Since $\varphi_w$ is a group homomorphism, for every $g_1,g_2\in F$ we have
$c_{1,n+1}(g_1g_2)=\sum_{k=1}^{n+1}c_{1k}(g_1)c_{k,n+1}(g_2)$, whence
\[
\begin{split}
(\partial c_{1,n+1})(g_1,g_2)&=-\sum_{k=2}^nc_{1k}(g_1)c_{k,n+1}(g_2)  \\
&=-\sum_{k=2}^n\bar c_{1k}(g_1F_{(n,R)})\bar c_{k,n+1}(g_2F_{(n,R)}).
\end{split}
\]
The explicit definition of the transgression map, as in \cite{NeukirchSchmidtWingberg}*{Prop.\ 1.6.5}, therefore implies that $\psi_w=-[\sum_{k=2}^n\bar c_{1k}\cup \bar c_{k,n+1}]=\trg[c_{1,n+1}|_{F_{(n,R)}}]$.
\end{proof}

We now define a homomorphism
\[
\Psi_{(n,R)}\colon \bigoplus_{{w\in X^*}\atop{|w|=n}}R \to H^2(F/F_{(n,R)},R), \quad
(r_w)_w\mapsto\sum_wr_w\psi_w.
\]

In view of Corollary \ref{restriction of mu hom} and Proposition \ref{inclusion of intersections},  there is a bilinear map
\begin{equation}
\label{upper pairing}
(\cdot,\cdot)''\colon \ F_{(n,R)}/F_{(n+1,R)}\ \times\  \bigoplus_{{w\in X^*}\atop{|w|=n}}R \to R, \quad
(\bar g,(r_w)_w)''=\sum_{|w|=n}r_w\varphi_w(g)_{1,n+1}
\end{equation}
with a trivial left kernel.
Combining this with  (\ref{cohomological perfect pairing}) (for $N=F_{(n,R)}$), we obtain a  diagram of bilinear maps of $\dbZ$-modules
\begin{equation}
\label{commutative bilinear maps}
\xymatrix{
F_{(n,R)}/F_{(n+1,R)} &  *-<3pc>{\times} &  \bigoplus_{{w\in X^*}\atop{|w|=n}}R \ar[r]^{\quad(\cdot,\cdot)''}\ar[d]^{\Psi_{(n,R)}} & R \ar@{=}[d]\\
F_{(n,R)}\ar@{->>}[u] &   *-<3pc>{\times} & H^2(F/F_{(n,R)},R)\ar[r]^{\qquad \quad (\cdot,\cdot)'} & R.
}
\end{equation}
This diagram commutes, in the sense that for $g\in F_{(n,R)}$ and for $(r_w)_w\in \bigoplus_{{w\in X^*}\atop{|w|=n}}R$ one has,
by Lemma \ref{psi as transgression},
\[
\begin{split}
(g,\Psi_{(n,R)}((r_w)_w))'&=(g,\sum_wr_w\psi_w)'=\sum_wr_w(g,\trg[c_{1,n+1}|_{F_{(n,R)}}])'\\
&=\sum_wr_wc_{1,n+1}(g)=\sum_wr_w\varphi_w(g)_{1,n+1}
=(\bar g,(r_w)_w)''.
 \end{split}
 \]

We denote
\[
H^2(F/F_{(n,R)},R)_{n-{\rm Massey}}=\Img(\Psi_{(n,R)}).
\]

\begin{thm}
\label{duality}
For every integer $n\geq2$ there is a canonical non-degenerate bilinear map
\[
F_{(n,R)}/F_{(n+1,R)} \times H^2(F/F_{(n,R)},R)_{n-{\rm Massey}}\to R.
\]
\end{thm}
\begin{proof}
This follows formally from the commutativity of (\ref{commutative bilinear maps}), and the triviality of the left- (resp., right-) kernel of the upper (resp., lower) bilinear map (see \cite{Efrat14a}*{Lemma 2.2}).
\end{proof}

\begin{exam}
\rm
Take $R=\dbZ$.
We recall that, by Magnus' theorem, $F_{(n,\dbZ)}$ is the $n$-th term $F^{(n,0)}$ in the lower central sequence.
Theorem \ref{duality} therefore gives a canonical non-degenerate bilinear map
\begin{equation}
F^{(n,0)}/F^{(n+1,0)}\times H^2(F/F^{(n,0)},\dbZ)_{n-{\rm Massey}}\to\dbZ.
\end{equation}
By the classical results of Magnus and Witt \cite{SerreLie}*{Part I, Ch.\ IV, Cor.\ 6.2}, $F^{(n,0)}/F^{(n+1,0)}$ is a free $\dbZ$-module of rank $l_{|X|}(n)$;
here $l_m(n)$ is the \textsl{necklace function}, defined for a positive integer $m$ by
\[
l_m(n)=\frac1n\sum_{d|n}\mu(d)m^{n/d},
\]
where $\mu$ is the M\"obius function, and with the convention $l_\infty(n)=\infty$.
We deduce:

\begin{cor}
\label{rank of H2}
For $n\geq2$, the $\dbZ$-module $H^2(F/F^{(n,0)},\dbZ)_{n-{\rm Massey}}$ is free of rank $l_{|X|}(n)$.
\end{cor}

For $n=2$ the quotient $\bar F=F/F^{(2,0)}$ is the free abelian group on basis $X$.
Then Corollary \ref{rank of H2} recovers the well-known fact that the image of the cup product map $H^2(\bar F,\dbZ)\tensor_\dbZ H^2(\bar F,\dbZ)\to H^2(\bar F,\dbZ)$ is a free $\dbZ$-module of rank $l_{|X|}(2)=\binom{|X|}2$ ($\infty$ if $X$ is infinite).
Indeed, the $2$-fold Massey product is the cup product, and the cohomology ring $H^*(\bar F,\dbZ)$ is the exterior algebra over $H^1(\bar F,\dbZ)\isom\oplus_X\dbZ$.
\end{exam}

\begin{exam}
\rm
Let $R=\dbF_p$ with $p$ prime.
Recall that, $F_{(n,\dbF_p)}=\mu_{\dbF_p}\inv(1+\grd_{\dbF_p}^n)$ is the $n$th term $F_{(n,p)}$ in the $p$-Zassenhaus filtration of $F$ (see (I)$_p$ of the Introduction).
We obtain for $n\geq2$ a non-degenerate bilinear map
\[
F_{(n,p)}/F_{(n+1,p)}\times H^2(F/F_{(n,p)},\dbF_p)_{n-{\rm Massey}}\to\dbF_p.
\]
A profinite  variant of this result  was proved in \cite{Efrat14a}*{Th.\ B}.

\end{exam}

\end{document}